\documentclass[10pt]{amsart}
\usepackage{amsmath,amsfonts,amsthm,amstext,amscd}

\usepackage[dvips]{graphicx}
\usepackage{array}
\usepackage{hyperref}
\usepackage{newlfont}
\usepackage{amsrefs}

\def\bb#1{\mathbb{#1}}

\def\lr#1{\left\langle #1\right\rangle}
\usepackage[all]{xy}

\newtheorem{mainthm}{Theorem}
 \newtheorem{thm}{Theorem}[section]
 
 \newtheorem{lem}[thm]{Lemma}
 \newtheorem{prop}[thm]{Proposition}
 \theoremstyle{definition}
 \newtheorem{defn}[thm]{Definition}
 \theoremstyle{remark}
 \newtheorem{rem}[thm]{Remark}

 \newtheorem*{ex}{Example}
 \numberwithin{equation}{section}
 \DeclareMathOperator{\Hom}{Hom}

\begin{document}

\title[Curvature of Pull-back-bundles]
 {Flat sections and non-negative curvature in pullback bundles}

\author[Dur\'an]{C. Dur\'an}

\address{C. Dur\'an\\Departamento de Matem\'atica, UFPR \\
 Setor de Ci\^encias Exatas, Centro Polit\'ecnico, \\
 Caixa Postal 019081,  CEP 81531-990, \\
 Curitiba, PR, Brazil}

\email{cduran@ufpr.br}

\author[Speran\c{c}a]{L. D. Speran\c ca}

\address{L. D. Speran\c{c}a\\IMECC - Unicamp,
Rua S\'ergio Buarque de Holanda\\
13083-859\\
Cidade Universit\'aria, Campinas \\
Brazil}

\email{llohann@ime.unicamp.br}

\thanks{The second author was financially supported by FAPESP, grant numbers 2009/07953-8 and 2012/25409-6.}

\subjclass{Primary  53C20}

\keywords{Non-negative curvature, principal bundles}


\begin{abstract}We give 
a geometric obstruction to the non-negativity of the sectional curvature in the total spaces of certain Riemannian 
submersions with totally geodesic fibers; applications of this obstruction to several examples are given.

  \end{abstract}

\maketitle
\section{Introduction}

 The study of manifolds which admit a metric of positive or non-negative sectional
curvature 
has lead to certain standard constructions and conditions that guarantee 
that these constructions furnish a metric with the desired properties. The classical 
example is the following: let $M$ be a manifold of positive sectional curvature, $G$ a compact Lie group
and $G\dots P\to M$
be a principal bundle-with-connection  over $M$. 
One can then endow $P$ with a ``Kaluza-Klein"-type metric, making the vertical and horizontal space orthogonal by definition, and inducing 
the metric on the horizontal space by the metric of $M$ and on the vertical space by some canonical metric on $G$ (typically biinvariant).
This procedure makes the bundle a Riemannian submersion 
with totally geodesic fibers.  O'Neill theory then says that the geometry of the total 
space if mostly controlled by the O'Neill tensor $A$ and the metric on the base (\cite{gromoll-walschap,oneill}). 

An immediate necessary condition for the positivity of curvature is that for $X$ vertical and $U$ horizontal, the 
O'Neill tensor $A^\ast_XU$ cannot be zero, since the vertizontal (unnormalized) sectional curvatures are given by 
$|A^\ast_XU|^2$ in this case. This condition, called {\em fatness} \cite{weinstein}, depends only on the connection on the bundle, and imposes upon it strong  
topological restrictions
\cite{derdzinski-rigas, florit-ziller}. Assuming fatness, 
in \cite{Chaves-Derdzinski-Rigas}
conditions are given for the total spaces of these principal bundles to admit such metrics 
of positive 
curvature.  These conditions take the form of differential inequalities relating the curvatures 
of $M$ and the bundle connection.

In this paper we are interested in {\em non-negative} curvature. 
The case of non-negative curvature on compact Riemannian manifolds still has many unanswered questions: 
for example, do 
all 7-dimensional exotic spheres or any sphere of other dimension admit a metric of non-negative curvature? (it is known to be true 
for exotic 7-spheres which can be realized as sphere bundles over $S^4$, \cite{grove-ziller}, 
and to be false for spheres which do not bound \emph{spin} manifolds, see \cite{licherowicz}, for example).
For non-negative curvature, the fatness condition on the O'Neill tensor can in principle be relaxed, but we shall see that 
 there  is a lot of rigidity along the degenerate directions.


\begin{mainthm}\label{main2}
Let $\xi:E{\to} M$ be a principal Riemannian submersion with totally geodesic fibers. Then, the set of vectors
\begin{gather}
D_b=\{X\in T_mM~|~A_{\hat X}=0, \,\, \hat X \,=\,\mathrm{horizontal \,\,lift\,\, of}\,\, X\}
\end{gather}
defines a smooth involution $D$ in an open and dense subset of $M$. Furthermore, a necessary condition for 
the total space $E$ to have non-negative curvature is that the integral manifolds of $D$ are totally geodesic.
\end{mainthm}

 In the generic case, theorem \ref{main2} will give the trivial foliation of $B$ by points. Therefore this theorem 
should not be seen as a ``structure theorem for manifolds of non-negative curvature''; instead, its value is in the 
construction of new examples of manifolds with non-negative curvature (which always use highly non-generic set-ups), 
as a quick way to check if there is any hope in a given metric. Section \ref{examples} will clearly illustrate 
this point. 
Of particular interest are {\em pullback} bundles;
abstractly, because all principal bundles are topologically pullbacks of the universal bundles, and also 
concretely, we have found that  many of the relevant examples in the literature of non-negative curvature are 
actually 
modelled as pullbacks,
and the natural induced connection and $A$-tensor will be degenerate along the fibers of the 
pullback map. 
The Gromoll-Meyer construction of an exotic sphere as a quotient of $Sp(2)$ falls within the 
scope of this result, since $Sp(2)$ with its canonical biinvariant metric is the pullback of 
the Hopf bundle (example \ref{sp(2)-as-pullback}). It has been recently realized that by 
studying other pullback maps many further geometric models of exotic spheres can be constructed 
(see \cite{DPR,lohan-8} and section \ref{examples}).

We can specialize theorem \ref{main2} to the case of a pullback connection:

 \begin{defn} \label{chubby}
 A Riemannian submersion is {\em non-degenerate} 
 if the map $X \mapsto A_X$ is a one-to-one mapping from the horizontal space 
into $\Hom(\text{Horizontal},\text{Vertical})$. 
 \end{defn}

Like fatness, non-degeneracy only depends on the connection form of the bundle. Non-degeneracy is a much weaker 
condition; 
fatness means  that for each 
horizontal $X\neq 0$ the map $A_X:\text{Horizontal}\to\text{Vertical}$ is onto and, in particular, non-zero.

 \begin{mainthm} \label{main-kk}
Let $\xi:G\cdots P \to B$ be a non-degenerate principal Riemannian submersion with totally geodesic fibers, 
$M$ a Riemannian manifold, and $f:M\to B$ a differentiable function. Endow the total space $E$ of the pullback bundle 
$f^\ast\xi$ with the 
Kaluza-Klein metric relative to the pullback connection. A necessary condition for
$E$ to have non-negative curvature
is that for each regular value $b \in B$, $f^{-1}(b)$ is totally geodesic in $M$.
\end{mainthm}

When the bundle $\xi$ is some canonical construction such as Hopf bundles,
the easiest way to represent a particular Kaluza-Klein metric is to consider the total space of 
$f^\ast\xi$ as a Riemannian submanifold of $M\times P$. This imposes a change on the metric on $M$, but we will see
(lemma \ref{tanto-faz}) that this is inconsequential.

\begin{mainthm} \label{main}
Let $\xi:G\cdots P \to B$ be a non-degenerate principal Riemannian submersion with totally geodesic fibers, $M$ a 
Riemannian manifold, and $f:M\to B$ a differentiable function. A necessary 
condition for the total space of the pullback $f^*\xi$ to have non-negative curvature with the pullback metric 
is that for each regular value $b \in B$, $f^{-1}(b)$ is totally geodesic in $M$.
\end{mainthm}

The condition of each regular fiber of a map $f:M \to B$ being totally geodesic places strong restrictions on 
the metric of $M$; typically, there will be no metric on $M$ for which the regular leaves are totally geodesic. 
A key element of this analysis is the study of how the map $f$ behaves near the singular leaves. We offer 
here a simple stability lemma that is enough to deal with the examples; a finer result will be given in 
\cite{stability-paper}. 

\begin{mainthm}\label{stability} Let $f:M\to B$ be a map such that the regular fibers are totally geodesic. 
If $S$ is a submanifold contained in  a (possibly singular) fiber of $f$  then  fibers  contained 
in a small enough neighbourhood of $S$ immerse into $S$.
\end{mainthm}

Therefore, if one suspects that the total space of a pullback bundle admits non-negative curvature, the 
first test should be to look at the behavior of the fibers of the pullback map on the base spaces, near singular points. 
We shall see in section \ref{examples} how this test discards some examples (in which a direct computation 
of the curvature would be rather involved), and give hope for others.

\bigskip

\noindent{\bf Notation}: Given a Riemannian metric, $\nabla$ denotes its Levi-Civita connection, 
$R$  the curvature tensor 
$R(X,Y)Z = \nabla_X\nabla_Y Z - \nabla _Y \nabla_X Z - \nabla_{[X,Y]} Z$, and $K$ 
the unnormalized sectional curvature 
associated to $R$. When necessary, left 
superscripts will be added to identify different spaces, e.g. $^M \nabla$.

In a Riemannian submersion, $P \stackrel{\pi}{\to} B$,  given $Y_x \in T_bB$ we denote its horizontal 
lift at $p\in \pi^ {-1}(y) \subset P$ by a hat, $\hat Y_p \in T_pP$. 
 The O'Neill tensor will be denoted by the letter $A$. Anywhere we give a property of a vector 
 $X \in T_bB$ in terms of a horizontal lift $\hat{X}$ in $T_{p}P$, it will be clear that 
the property is independent of the point $p \in \pi^{-1}(b)$.

 \section{Geometry of pullback bundles} \label{geometry-of-pullbacks}

Consider a principal bundle $G \cdots P \stackrel{\pi}{\to} B$. Let $M$ be a Riemannian manifold, $f: M \to B$ a differentiable map and consider 
 the pullback $G\cdots f^*P \to M$. 
 
 \begin{defn} The {\em pullback metric} on $f^*P$ is the metric induced 
 as a submanifold by the definition of a pullback, i.e., 
$f^*P=\{(m,p)\in M\times P \,\,|\,\, f(m) = \pi(p)\}$, where we endow $M\times P$ with the product metric. 
\end{defn}

\begin{ex} \label{sp(2)-as-pullback}

The Lie group $Sp(2)$ can be realized as the total space of the  pullback of the Hopf bundle by  minus the Hopf map: 

\[
\begin{CD}
@. @. S^3 @. S^3\\
@. @. \vdots @. \vdots\\
S^3 @.\cdots @. {Sp(2)} @> s >> {S^7} \\
@. @.  @ V{f} VV @ VV {h} V \\
S^3 @.\cdots @. {S^7} @>> {a\circ h} > {S^4} \\
\end{CD} 
\]

In the diagram, $h$ denotes the Hopf map corresponding to the right Hopf action of $S^3$ on $S^7$, $a$ the antipodal map of $S^4$, and $f$ and $s$ denote the first 
and second columns of a matrix in $Sp(2)$. 
When both $S^7$'s in the diagram have the same constant curvature, then the pullback metric of 
$Sp(2)$ is biinvariant (see \cite{rigas-BSMF} and proposition 2.1 in \cite{wilhelm-lots}).

\end{ex}

 Some more sophisticated examples will be considered in section \ref{examples}.

Let us establish some geometrical properties of the pullback bundle. First, the submersion 
$G \cdots f^*P \to M$, where $f^*P$ has the pullback metric and $M$ 
has its given metric, is {\em not} a Riemannian submersion. It is easy to see that in order 
to make it a Riemannian submersion, we must endow $M$ with the {\em graph metric} induced 
from the embedding of $M$ as a subset
$\Gamma_f = \{(m,b) \in M\times B \,\, |\,\,f(m) = b\} \subset  M \times B $, 
the later endowed with the product metric. Indeed,

\begin{prop}
Let $\pi:P\to B$ be a principal bundle with the Kaluza Klein metric defined by 
a connection 1-form $\omega$ and a metric $g_B$ on $B$, $M$ a Riemannian manifold, 
and $f:M \to B$  a smooth map. Then, 
the pullback metric on $f^*P$ is the Kaluza-Klein metric defined by the 
pullback connection $f^*\omega$ and the graph metric $g_{\Gamma_f}$.\end{prop}

\begin{proof}

Let $W=(X,Y)$ be a vector tangent to $f^*P$. 
So, by the definition of $f^*P$, $df(X)=d\pi(Y)$. Spelling out the induced metric on $f^*P$, we have
\begin{align*}
g_{f^*P}(W,W)&=g_{M\times P}((X,Y),(X,Y))=g_M(X,X)+g_P(Y,Y)\\
&=g_M(X,X)+g_B(d\pi Y,d\pi Y)+\beta(\omega(Y),\omega(Y))\\
&=g_M(X,X)+g_B(dfX,dfX)+\beta(\omega(Y),\omega(Y))\\
&=g_{\Gamma_f}(X,X)+\beta(\omega(Y),\omega(Y))
\end{align*}

Now the proof follows by observing that $Y$ is the image of $W$ by the 
derivative of the induced pullback map. In particular, $\omega(Y)$ is the image of 
$W$ by the pullback connection, as desired.

\end{proof}

\begin{rem}
The fact that $M$ has to change its original metric in order 
for $f^* P \to M$ be a Riemannian submersion is already present
in the example from the introduction: with a biinvariant metric the submersion 
$Sp(2) \to S^7$ subduces a metric on $S^7$ that is not the 
round one, a fact used in \cite{duran}.
\end{rem}

The vertical space is given by vectors of the form $(0,U)$, $U\in TP$ vertical. 
Horizontal vectors have the form $(Z,\widehat{df(Z)}), Z \in TM$. In particular, if $Z$ is tangent to a level set of $f$, 
then $(Z,0)$ is horizontal.

The following property will be essential in theorem \ref{main}:

\begin{lem} \label{tanto-faz}
 Let $M,B$ be Riemannian manifolds,  $f: M \to B$, and $b\in B$ be a regular value of $f$. Then $f^{-1}(b)$ is totally geodesic 
in $M$ if and only if it is totally geodesic in the graph $\Gamma_f$. 
\end{lem}

\begin{proof}
On the one hand, a regular level set $L=f^{-1}(b)\subset M$ is totally geodesic  on $M$ if and only if, 
given $X\in TL$ such that $df(X)=0$, 
$df(^M \nabla_X X)=0$. On the other hand, via the diffeomorphism between $M$ and the graph, $m \stackrel{\psi}{\mapsto} (m,f(m))$, a vector field tangent 
to $L$ is translated to a field $\hat X = (X,0)$ on $M\times B$ such that $df(X)=0$. We have 
\begin{gather}\label{eq 1}
^{\Gamma_f}\nabla_{\hat X}{\hat X}=proj_{T\Gamma_f}( ^{M\times B} \nabla_{(X,0)}(X,0)) = proj_{T\Gamma_f}(^M\nabla_X X, 0) \, .
\end{gather}
Furthermore, $\psi(L)=L\times\{b\}$ is totally geodesic on $\Gamma_f$ if and only if 
the right-hand side of \eqref{eq 1} is of the form $(Y,0)$. 

In particular, if $L$ is totally geodesic on $M$, it is totally geodesic on $\Gamma_f$. On the other hand, decomposing
this vector in tangential and normal part we have
\[(^M\nabla_X X, 0)=(W,df(W)) + (-df^*(Z),Z)\] 
with $df(W)=-Z$. So, $proj_{T\Gamma_f}(^M\nabla_X X, 0)=(Y,0)$ if and only if $Y=W$ and $df(Y)=0$, in particular, $Z=-df(Y)=0$ and $(^M\nabla_X X, 0)$ 
is already tangent, i.e., $df(^M\nabla_X X)=0$.

\end{proof}

\section{Flat vertizontal sections}\label{flat-leaves}

In this section,  $G \cdots E \to M$ will be a  
Riemannian submersion with totally geodesic fibers 
such that $E$ has non-negative curvature. 

We first need the following elementary lemma on flat sections on non-negatively curved spaces:

\begin{lem}\label{curvature lemma}
 Let $(Q,g)$ be a non-negatively curved manifold and suppose the plane spanned by $X$ and $U$  has zero sectional 
curvature. Then $^Q R(X,U)X = 0$. 
\end{lem}

\begin{proof}
This follows immediately by the well-known ``quadratic trick'' (compare \cite{Chaves-Derdzinski-Rigas}): 
let $Z$ be an arbitrary vector. Consider the sectional curvature of the plane spanned by $X$ and $tU+Z$. 
Unnormalized, this is 
 \[
  g(R(X , tU + Z) tU+Z, X) = t^2 K(X, U) -  
 2tg (R( X, U) X, Z) + 
 K( Z, U) \, .
 \]
If $K(X,U)=0$  then the quadratic term vanishes and the resulting expression is a non-negative 
linear function. Thus 
we get that $K \geq 0$ implies that the linear coefficient 
$g(R( X, U)X ,Z) = 0$ for all $Z \in TQ$. 
\end{proof}

\begin{rem} The condition being symmetric in $X$ and $U$, the conclusion also holds interchanging $X$ and $U$. 
 This means that what really happens is best explained in terms of the curvature operator 
$\mathcal R: \Lambda^2(TQ) \to \Lambda^2(TQ)$: given $\eta \in \Lambda^2(TQ)$ we define 
$\phi_\eta: \Lambda^2 (TQ) \to \Lambda^4 (TQ)$, $\phi_\eta(\xi) = \eta \wedge \xi$. Then 
if $K(X,U)=0$, then $\mathcal R(X\wedge U) \in \ker \phi_{X\wedge U}$. This gives some rigidity 
to zero curvature sections, an example of which is the main result of this paper (see also proposition 1 in 
\cite{wilhelm-flat}).
\end{rem}

\begin{rem} One may notice that an analogous proof works in the non-positive curvature case. In fact, being here where the non-negativity requirement enters, our main results work replacing non-negatively curved by non-positively curved.\end{rem}

\noindent {\it Definition}. A vector $X \in  TM$ is said to be {\em $A$-flat} if 
$A_{\hat X} = 0$.  An $A$-flat vector $X$ is called regular if it can be locally extended to a 
vector field of $A$-flat vectors; we call such fields $A$-flat vector fields. 

We abuse notation by also calling $A$-flat, the horizontal, basic vectors $\hat X$ such that $A_{\hat X} = 0$. 
O'Neill equations imply that a vector is $A$-flat if and only if the curvature of all vertizontal planes containing 
$\hat X$ vanish. 

\begin{thm} \label{A-flat-accel}
 If $X$ is an $A$-flat vector field, then so is the covariant acceleration $\nabla_X X$.
\end{thm}

\begin{proof}
 Let $Z$ be another vector field on $M$, and $U$ a vertical field on $E$.

Spelling out the curvature term in \eqref{curvature lemma} and using the identities in \cite{gromoll-walschap}, we have:
\begin{align*}
0= g_P(R(\hat X, U)\hat X , \hat Z)&=-\lr{\nabla_{\hat X}(A_{\hat X}\hat Z)	,U}+\lr{A_{\nabla_{\hat X}\hat X}\hat Z,U}+\lr{A_{\hat X}\nabla_{\hat X}\hat Z,U}\\
&=\lr{A_{\nabla_{\hat X}\hat X}\hat Z,U}
\end{align*}
and since the equality holds for all $\hat Z$ and $U$, $A_{\nabla_{\hat X}\hat X}=0$. The theorem follows from the identity $\widehat{\nabla_XX}=\nabla_{\hat X}\hat X$.\end{proof}

\noindent {\it Definition}. A submanifold $S$ of $M$ is said to be  $A$-flat 
if vectors tangent to 
$S$ are $A$-flat. It is maximal if for all  $s\in S$ and $X \in T_sM$, $X$ $A$-flat implies that 
$X \in T_sS$. 

Then, it immediately follows from theorem \ref{A-flat-accel} that 

\begin{thm} \label{A-flat-totally-geodesic} 
 A maximal $A$-flat submanifold $S$ is totally geodesic in $M$. 
\end{thm}

\begin{rem} Note that the proof of theorem \ref{A-flat-accel} relies on a tensorial identity, and theorem 
\ref{A-flat-totally-geodesic} uses only information of the vector field $X$ along the submanifold $S$. 
Therefore, theorem \ref{A-flat-totally-geodesic} holds without needing to extend $X$ to an open neighborhood 
of $S$ in $M$. This will be useful on example \ref{rigas-bundles}.\end{rem} 

Now we see that there are plenty of maximal $A$-flat submanifolds: 

\begin{lem}\label{involutive}
 The set of all $A$-flat vector fields is in involution.

\end{lem}

\noindent{\it Proof.} A vector field $X$ on $M$ is $A$-flat if and only if
$g_P([\hat X, \hat Z], U) = 0$
for all vector fields  $Z$ on $M$ and $U$ vertical vector field on $E$. Now the result follows from
Jacobi's identity and the fact that $[\hat X, \hat Z]=\widehat{[X,Z]}$ when $A_{\hat X}=0$. \qed

The $A$-flat vectors on the regular part of the distribution are regular.
On this set, we can integrate the 
distribution spanned by the $A$-flat vectors in $M$. We call such integral manifolds
{\em $A$-flat leaves}. Then by theorem \ref{A-flat-totally-geodesic}, we have

\begin{thm} \label{flat-leaves-totally-geodesic}
 Assume $E$ has non-negative curvature. Then the  $A$-flat leaves are totally geodesic in $M$.  
\end{thm}

For the conclusion of the proof of theorem \ref{main2}, what is missing is regularity considerations. 
Let $D_b$ be as in Theorem \ref{main2}, i.e., the set of all $A$-flat vectors, then $M$ is partitioned according 
to $0 \leq \dim D_b \leq \dim M$. It is easy to see that the set of points in $M$ such that $\dim D_b$ is locally
{\em minimal} is open and dense. Indeed, $D_b$ is the kernel of a continuous family of linear maps parametrized by 
$m\in M$, and, as such, the function $d:M\to \mathbb{N}$, $d(m)= \dim(D_m)$ is lower semicontinuous. Thus the set 
of local minima, on which the rank of $D_m$ is locally constant
 and $D_m$ is a regular foliation, is open and dense (compare \cite{lewis}). 

Now let us deduce theorems \ref{main-kk} and \ref{main} from theorem \ref{main2}: 

Recall that if $G\cdots E \to M$ is a principal Riemannian submersion, 
then, for horizontal fields $X,Z$ and $U$ vertical,  the 
O'Neill tensor $g_P(A_X U, Z) = - b_G(\Omega(X,Z),\xi_U)$, 
where $\Omega$ is the curvature form of the 
bundle connection induced by the horizontal-vertical decomposition, $b$ is a biinvariant metric on $G$ and $\xi_U \in \mathfrak{g}$ is the infinitesimal generator of $U$. Thus, 
the condition of being $A$-flat can be completely written in terms of the curvature of the 
connection form of the bundle: it is in the kernel $\{X \in T_mM \,\,: \,\, \Omega(X,Z) =0 \,\,\forall\,\, Z\}$

Consider now a pullback $G\cdots f^*P \to M$ of a non-degenerate principal Riemannian submersion 
$G \cdots P \to B$, with 
bundle curvature form $\Omega$. Since the main theorem concerns regular values $b \in B$, by restricting 
to the open set of regular values, we might assume that $f$ is a submersion.
 
The curvature form of the pullback bundle is just the pullback form 
$f^*\Omega$. Thus $X \in T_mM$ is $A$-flat if and only if $\Omega(df_m(X),df_m(Z)) = 0$ for all $Z\in T_mM$, and 
since $f$ is a submersion then for all $Y\in T_{f(b)}B$ there exists $Z \in T_mM$ such that $df_m(Z) = Y$. This 
means that  $\Omega(df_m(X),Y) = 0$ for all $Y\in T_f(m)B$, and since by hypothesis $G \cdots P \to B$ is non-degenerate, 
it follows that $df_m(X) = 0$. We have thus concluded that $A$-flat vectors are tangent to level sets of $f$, and maximal $A$-flat submanifolds 
and $A$-flat leaves
are exactly the fibers of $f$. 
Then theorem \ref{A-flat-totally-geodesic} implies that the fibers of 
$f$ are totally geodesic on $M$. 

For theorem \ref{main}, put the graph metrics on the total space and on $M$. Then we are in a Kaluza-Klein 
construction and we can apply theorem \ref{main-kk}
 {\em with respect to the graph metric} of $M$.  However,  lemma \ref{tanto-faz} then 
finishes the proof.

\section{Stability of regular totally geodesic fibers}

In this section we prove theorem \ref{stability}.

Let $S\subset M$ be a compact submanifold, possibly with boundary, and $\nu$ its normal bundle, with $\delta$-disk 
bundle $\nu_\delta$. For $\delta$ small enough, the exponential map $\exp:\nu_\delta \to M$ is a diffeomorphism 
onto a tubular neighbourhood $U_\delta \subset M$. There is radial projection $p:U_\delta \to S$, given 
by $u \mapsto p(u)$, the (unique) closest point on $S$ to $U$. With this context, we have

\begin{prop} \label{pre-stability}
There exists $\delta>0$ such that, if $L$ is a complete totally geodesic submanifold contained in $U_\delta$, 
then $p|_L:L\to S$ is an immersion.
\end{prop}

Note that the kernel of $Dp$ is given by vectors tangent to the fibers $F^\delta_x=\exp_x(\nu_\delta)$. Since $L$ is 
totally geodesic and complete, proposition \ref{pre-stability} is reduced to the following:

\begin{lem}\label{lem stab}
 There exists $\delta>0$ such that, if $m\in U_\delta$ and $v_m\in \ker Dp_m$, 
then $\exp_m tv\notin U_\delta$ for some $t$.\end{lem}

\begin{proof}[Proof of Lemma \ref{lem stab}:] 

The idea is simply to observe that, close enough to $S$, we are 
essentially in an Euclidean situation illustrated by, say, $S$ being the $z$ axis in $\bb R^3$ and $v=(v_1,v_2,0)$ a 
vector tangent to $\bb R^2 \times \{z_0\}$ at a point $(x_0,y_0,z_0)$. Let us proceed: denote by $\eta:M\to \bb R$  the map defined by  half of the square  of 
 the distance to $S$, i.e., $\eta(x)=\frac{1}{2}d(x,S)^2$. The following facts are either elementary or 
well-known: (e.g. \cite{ambrosio-mantegazza}): $\eta$ is smooth on $U_{\epsilon}$, and  for $x\in S$, its derivatives 
satisfy
\begin{itemize}
 \item The derivative $d\eta_x = 0$.
\item The Hessian $d^2\eta_x(v,w) = g(\Pi v,\Pi w)$, where $g$ is the metric on $M$ and $\Pi$ is the orthogonal projection onto the 
normal bundle of $S$. Note that since $d\eta_x=0$, at points of $S$ $d^2\eta$ is a true (e.g. independent of the 
metric) Hesssian. 
\item $d^3\eta_x$ is  linearly expressed in terms of the second fundamental form of $S$.
\end{itemize}
Also, the second fundamental form of the fibers $F^\delta_x$ is zero at $x$. Extend $\Pi$ of the previous bullet to 
represent the orthogonal projection onto the tangent bundle of the fiber $F_x^{2\delta}$ at any point. 
Now given $\epsilon>0$, we can uniformly choose 
$\delta$ such that, for $d(y,S)<2\delta$, we have that 
\begin{itemize}
 \item The norm of the second fundamental form of each fiber $F^{2\delta}$ is less than $\epsilon$. This implies that 
by choosing $\epsilon$ wisely, we can assure that
for any unit geodesic $\gamma(t), t\in [0,1]$ with inital velocity vector tangent to the fiber, then 
\[
 g(\Pi \dot{\gamma}(t),\Pi \dot{\gamma}(t)) > g((\mathbf{1} - \Pi) \dot{\gamma}(t),(\mathbf{1} - \Pi) \dot{\gamma}(t))\, .
\]
(or, equivalently $2g(\Pi \dot{\gamma}(t),\Pi \dot{\gamma}(t)) > g(\dot{\gamma}(t), \dot{\gamma}(t))$)
for as 
long as $\gamma(t) \in U_{2\delta}$. Since 
\item The Hessian satisfies 
\[
 d^2\eta(v,v) \geq \frac{9}{10}g(\Pi v,\Pi v) - \frac{1}{10}g((\mathbf{1} - \Pi) v,(\mathbf{1} - \Pi) v)\, .
\]

\end{itemize}

Consider now the function $h(t)=\eta(\gamma(t))$, for a unit geodesic $\gamma$ tangent to the normal fiber and such 
that $0\leq d(\gamma(0),S)< \delta$, which translates to $0\leq h(0)<\frac{\delta^2}{2}.$ By reversing the orientation 
of $\gamma$ if neccessary, we can assume that $h'(0) \geq 0$. Now 

\begin{eqnarray*}
 h''(t) &=& d^2\eta_{\gamma(t)}(\dot{\gamma}(t), \dot{\gamma}(t)) \\
        &\geq&  \frac{9}{10}g(\Pi \dot{\gamma}(t),\Pi \dot{\gamma}(t)) -
                \frac{1}{10}g((\mathbf{1} - \Pi)\dot{\gamma}(t),(\mathbf{1} - \Pi)\dot{\gamma}(t) )  \\
        &\geq&  \frac{4}{5}g(\Pi \dot{\gamma}(t),\Pi \dot{\gamma}(t)) \\
        &\geq& =  \frac{2}{5}g( \dot{\gamma}(t),\dot{\gamma}(t)) =  2/5 \, ,       
\end{eqnarray*}
for $t\in [0,1]$ and for as long as $\gamma(t)$ lies in $U_{2\delta}$. Thus we have a function $h$ such 
$h(0)\geq 0, h'(0)\geq 0$ and $h''(t) > 2/5$. Thus $h(t)\geq \frac{1}{5}t^2$, which means that for $\delta^2<1/5$ in addition 
of the previous conditions, the geodesic $\gamma(t)$ leaves $U_\delta$. \qed 
\end{proof}

Now theorem \ref{stability} is just stating proposition \ref{pre-stability} in the situation we are interested in, that is,
when $S$ is a subset of a fiber of $f:M \to N$. 

\section{Examples} \label{examples}

In this section we provide several examples where the obstruction applies in different ways. 

\begin{itemize}
\item In the first two examples, the main obstruction (theorem \ref{main})  vanishes for the canonical metrics. However, in the 
second example the ``secondary'' obstruction furnished by proposition \ref{A-flat-totally-geodesic} does not vanish for 
the canonical round metric on the base sphere; a deformation is necessary to make all obstructions vanish.  
\item In the third and fourth examples, the obstruction is absolute: we show that for the pullback maps presented, 
for any given metric the fibers cannot be totally geodesic. These examples arise as pullbacks of the Gromoll-Meyer construction, which 
we think in the context of the pullback of the Hopf map $h: S^7 \to S^4$ as in example \ref{sp(2)-as-pullback}; that 
is, we  pullback the  fibration $S^3 \cdots Sp(2) \to S^7$ by appropriate maps $f:X \to S^7$. However, 
since the canonical connection of $S^3 \cdots Sp(2) \to S^7$ is degenerate (being zero along the Hopf fibers), we go all the way 
to the Hopf bundle and pull back  $S^3 \cdots S^7 \to S^4$ by maps $f: X \to S^4$ of the form 
$f = a \circ h \circ \phi$, where $\phi:X \to S^7$ and $a$ is the antipodal map of $S^4$. All actions in sight will be isometric 
with respect to the induced connections and metrics.
In both examples, our results implies that one cannot put non-negative curvature in the total space of 
these bundles via the pullback construction.  

\item Finally, we describe the pullback structure of Kervaire spheres. Here,
in contrast with the previous examples, 
there is no {\em topological} obstruction to the presence of the totally geodesic foliation, however the authors have not 
been able to describe a metric with totally geodesic fibers in that case. 
\end{itemize}

\subsection{Wilhelm bundles} \label{wilhelm-bundles}
Consider the bundles constructed on \cite{wilhelm-lots}: let $Sp(2,m)$ be the subset of $m$ copies of $S^7$ defined by the following condition
\[Sp(2,m)=\{(u_1,...,u_m)\in (S^7)^m~ |~ h(u_1)=ah(u_2),~\tilde h(u_{i})=ah(u_{i+1})~for~i>1\}\]
where   $a$ and $h$ are as in example \ref{sp(2)-as-pullback} and 
$\tilde h$ is the dual Hopf map corresponding to the left $S^3$-action on $S^7$. 
Quotients of these bundles by free $S^3 \times \dots \times S^3$-actions 
give models of some 3-sphere bundles over $S^4$, and, in particular, exotic spheres. 

As we observed earlier, the pull-back diagram of  $Sp(2)$ endowed with its canonical metric does not have 
the obstruction given by Theorem \ref{main}. We also observe that $Sp(2,m+1)$ fits in the following diagram
\begin{equation}
\begin{xy}\xymatrix{Sp(2,m+1)\ar[d]\ar[r]^-{pr_{m+1}} & S^7\ar[d]^{ah}\\Sp(2,m)\ar[r]^-{\tilde h\circ pr_m}&S^4}\end{xy}
\end{equation} 
where $pr_m:Sp(2,m)\to S^7$ is the projection in the last coordinate and the unnamed 
down arrow is the projection in the first $m$ coordinates. Now, induction  on $m$ 
easily proves that $\tilde h\circ pr_m$ has totally geodesic fibers, as required from Theorem \ref{main}.

\subsection{Rigas bundles} \label{rigas-bundles}
Let $\phi_k:S^4 \to S^4$ be a degree $k$ map, and $r_k= \phi_k \circ h: S^7 \to S^4$. Then the bundles 
$\tilde P_k$ constructed in \cite{rigas-BSMF} are defined as the pullback of Hopf by $r_k$ (for example, 
if $\phi_{-1}$ is represented by the antipodal map of $S^4$, then $\tilde P_{-1}$ is
$S^3 \cdots Sp(2) \to S^7$ as in example \ref{sp(2)-as-pullback}). 
\begin{equation}
\begin{xy}\xymatrix{\tilde P_k \ar[d]\ar[r]^-{} & P_k\ar[r]\ar[d]& S^7\ar[d]^h\\S^7\ar[r]^-{h}&S^4\ar[r]^{\phi_k}&S^4}\end{xy}
\end{equation}

We choose $\phi_k:S^4 \to S^4$ to be the suspension of $p_k$, the quaternion $k$-th power map $S^3 \to S^3$ (note that, 
for $k=-1$, this is different from the antipodal map). We suspend this map in the simplest way, 
\[
S^4 
\supset
\bb R \times \bb H
\ni
 \begin{pmatrix}
  x \\ y 
 \end{pmatrix}
\stackrel{\phi_k}{\mapsto}
\frac{1}{\sqrt{x^2 + |y|^{2k}}}
 \begin{pmatrix}
  x \\ y^k
 \end{pmatrix} \, .
\]
 The critical values of maps $\phi_k$, and {\em a fortiori} $f_k$, since the Hopf map is a submersion,
 are given by $(\pm 1, 0)^\top$ and the suspensions of the meridians 
 $(x,\cos(\ell\pi/k) + \sin(\ell\pi/k) \hat \alpha)$, $\alpha$ a purely imaginary quaternion and $1<  \ell < k$. 
Restricted to the inverse images 
of the complement 
of such points, $\phi_k$ is a local diffeomorphism. Therefore, the regular fibers of $r_k$ will be given by 
sets of $k$ disjoint standard Hopf fibers, and the obstruction given by theorem \ref{main} also vanishes in this case for the canonical
round metric on $S^7$. 
However, if $|k|\geq 2$ the secondary obstruction given by theorem \ref{A-flat-totally-geodesic} does {\em not} vanish for the 
round metric on $S^7$ which projects by the Hopf map to the round metric on $S^4$, since the pullback by Hopf of the 
suspensions of the meridians $(x,\cos(\ell\pi/k) + \sin(\ell\pi/k) \hat \alpha)$ will not be totally geodesic. Note that, 
by continuity, the singular fibers detect negative curvature in the regular set that was invisible by just using 
theorem \ref{main}.

Therefore, if one wants to construct a pullback metric of non-negative curvature on 
the Rigas bundles, the metric on $S^4$ must be changed so that this suspended meridians are totally geodesic, 
(by making $S^3$ to be a cylinder $S^2 \times I$, in  a set that contains the meridians), and then the metric 
on $S^7$ is defined by a Kaluza-Klein procedure over the Hopf map.

\subsection{Exotic 7-spheres} \label{duran-puttman-rigas-bundles} Consider the map $\phi_n: S^7 \to S^7$ given by the $n$-th power of the 
Cayley octonions $\bb O$. If we write a unit octonion
$q = \cos(t) + \alpha\sin(t)$ where $\alpha$ is purely imaginary, then 
$\phi_n(t) = \cos(nt) + \alpha\sin(nt)$. Pulling back the bundle $S^3 \cdots Sp(2) \to S^7$ by $\phi_n$ we obtain principal $S^3$-bundles 
$S^3 \cdots E^{10}_n \to S^7$. Now writing octonion as pairs  $(a,b)^\top$ of (column) quaternions, then the total space is given 
by 
\[
 E_n \cong 
\left\{
\begin{pmatrix}
 a & c \\
b & d 
\end{pmatrix}
\,\,
\Big|
\,\,
\left(\hspace{-2pt}
\phi_n\hspace{-3.5pt} 
\begin{pmatrix}
 a \\
b
\end{pmatrix}~
\begin{matrix}
 c \\
d
\end{matrix}
\right)
\in Sp(2)
\right\} \, .
\]
With the projection onto $S^7$ given by the projection onto the first column $(a,b)^\top$.
The unit quaternions act on $E_n$, by 
\[
 q\star 
\begin{pmatrix}
 a & c \\
b & d 
\end{pmatrix}
=
\begin{pmatrix}
 q a \bar q & qc \\
q b \bar q & qd 
\end{pmatrix} \, .
\]

The following is proven in \cite{DPR}: the quotient of $E^{10}_n$ by this action is diffeomorphic to 
$\Sigma^7_n$, $n$-times the Gromoll-Meyer sphere in the group of 7-dimensional 
homotopy spheres. Going all 
the way to the Hopf bundle, consider $E^{10}_n$ as the pullback of the Hopf bundle by $f_n = a\circ h\circ\phi_n$, 
where $a$ and $h$ are as in examples \ref{sp(2)-as-pullback} and \ref{wilhelm-bundles}. 
\begin{equation}
\begin{xy}\xymatrix{\tilde E_n \ar[d]\ar[r]^-{} & Sp(2)\ar[r]\ar[d]& S^7\ar[d]^h\\S^7\ar[r]^-{\phi_n}&S^7\ar[r]^{a\circ h}&S^4}\end{xy}
\end{equation} 

\begin{rem}
 It is instructive to compare these spaces to Rigas' bundles on the previous section; here the power map is 
on the Cayley numbers as the last map of the composition giving the pullback, and there the power map (of the quaternions) 
is the first step in the composition. 
Both cases are pullbacks of the Hopf bundle over $S^4$ by functions $S^7\to S^4$, and the total spaces $\tilde P_k$ of  and 
$E_n$  are diffeomorphically related 
by the formula $E_n \cong \tilde P_k$, where $n = k(k+1)/2 \mod 12$ \cite{barros}. However, for the presentation 
given by the pullback by $r_k$  the main obstruction vanishes, and by modifying the metric all obstructions can be eliminated,
whereas we will presently see that this is impossible for the pullbacks by $f_n$:  no metric on $S^7$ makes the regular fibers totally geodesic, and thus any pullback metric on 
$E_n = f_n^\ast{\text{Hopf}} $ has some negative curvature. In principle, the O'Neill tensor of the exotic actions could push the curvature in the base exotic sphere to be 
non-negative, altough recent results 
suggests that the odds are not good \cite{wilhelm-flat}; this comment also applies to the exotic 8-spheres of next section.
\end{rem}

We have 

\begin{lem}
 The only singular value of $f$ is the south pole $(-1,0)\in S^4\subset \bb R \times \bb H$. 
\end{lem}

\noindent {\it Proof of lemma.} Observe that the only singular values of $\phi_n$ are $\pm 1 \in S^7 \subset \bb O$, on the preimage of 
which the derivative of $\phi_n$ has a kernel of dimension at least 6 ($\phi$ being constant on the distance spheres 
$\cos(\frac{k\pi}{n}) + \alpha \sin(\frac{k\pi}{n})$. Thus the only 
singular value of $h\circ \phi_n$ is the class of the Hopf fiber through 1, which maps to $(1,0) \in S^4\subset \bb R \times \bb H$.
The lemma now follows by composing with the antipodal map. \qed

\medskip

Restricted to the set  $\phi_n^{-1}(\pm 1)$, the map $\phi_n$ is a (trivial) $n$-fold covering. Thus, the inverse image 
of regular values in $S^4$ consists the inverse image of Hopf fibers by $\phi_n$, which are $n$ disjoint 
3-spheres, each one contained in a ``belt''
$\{ \cos(t) + \alpha \sin(t) \,\, | \,\, \frac{k\pi}{n} < t < \frac{(k+1)\pi}{n} \}$. 

The singular  set
$\phi^{-1}(-1,0)$ is the union of the exponential 3-sphere $(x,0), x$ a unit quaternion, 
with the distance spheres  $\cos(\frac{k\pi}{n}) + \alpha \sin(\frac{k\pi}{n})$.

We have now 

\begin{prop} \label{facemask}
if $n>1$, there is no metric on $S^7$ such that the inverse images of the regular values of $f_n$ are totally geodesic.
\end{prop}

The remaining of this example is devoted to the proof of proposition \ref{facemask}; the idea is to find a one 
parameter family of regular fibers of $f_n$ whose convergence to the singular fiber contradicts proposition \ref{stability}. 

 Consider the one-dimensional manifold $T$ of $S^4 \subset \bb R \times \bb H$ given by points having the second 
coordinate real,  
The inverse image of $T$ by the Hopf map (and also by $a\circ h$, since $T$ is invariant under the antipodal map) 
 is the elements of $S^7 \subset \bb H \times \bb H$ which have both coordinates linearly dependent over 
the reals. This condition is invariant by the Cayley power map, and thus the inverse image of $T$ by 
$f_n = a\circ h \circ \phi_n$ is also described by the real dependence of the first and second coordinates.

Consider now the curve $\sigma(\theta)$ in $T\subset S^4$  $\sigma(t)=(\cos(\theta),\sin(\theta))$, $\theta\in[\pi-\epsilon,\pi)$, the interval 
being chosen so that the trace of the curve curve is lies inside of the regular set, but approaches  the critical value 
$(-1,0)$ as $\theta\rightarrow \pi$.
 
The inverse image 
 $(a\circ h \circ\phi_n)^{-1} (\sigma(\theta))$ is made of $n$ disjoint 3-spheres, which we do not need to 
characterize precisely.  In order to apply proposition \ref{stability} we need to describe the evolution of 
these spheres as $\theta\rightarrow \pi$; we will also consider just the connected component of the preimage 
contained in the north polar cap $N = \{\cos(t) + \alpha\sin(t) \in S^7, \,| \, t\in [0,\pi/n]\}$.  

Let then $q$ be a unit Cayley number written as a pair of quaternions, $q = (a,b)^\top$, where 
$a = \cos(t) + \sin(t)p$  and 
 $b= \sin(t) w$, $|p|^2 + |w|^2=1$, $t\in [0,\pi/n]$.
We have
\[
f_n(a,b)^\top = (2 \sin^2(nt)|w|^2 - 1, \star) \, ,
 \]
 where the $\star$ in the second coordinate is determined by the first when we 
are close to $(-1,0) \in S^4$ and both coordinates are real. Given $\theta \in [\pi-\epsilon, \pi)$, the formula
for $f_n(a,b)^\top = \cos(\theta), \sin(\theta)$ implies that neither $w$,$\sin(nt)$ nor $\sin(t)$ are zero. Then
\[
 2 \sin^2(nt)|w|^2 - 1 = 2\frac{\sin^2(nt)}{\sin^2(t)}(\sin^2(t)|w|^2) -1 = \eta(t)|b|^2 -1 \, ,
\]
where $\eta(t) = 2\frac{\sin^2(nt)}{\sin^2(t)}$ is bounded away from zero in the interval $[0,\pi/n]$. Then 
as $\theta \rightarrow \pi$, $|b| \rightarrow 0$. This means that the 3-spheres 
$f_n^{-1}(\cos(\theta),\sin(\theta)) \cap N$ converge to a set $S$ contained in the meridian $(a,0), |a|=1$, 
and being contained in the north polar cap we can also say that  
$S \subset \{ (a,0) \in S^7 \, | \, \Re(a) < \frac{3\pi}{2n}\}$. This set is diffeomorphic 
to an open ball in $\bb R^3$. Thus, proposition \ref{stability} would furnish an embedding of 
a 3-sphere into a 3-ball in $\bb R^3$, which is impossible. \qed

\subsection{Exotic 8-sphere} \label{llohan-bundles} Consider the map $\phi:S^8\to S^7$ given by suspending the Hopf map $\eta:S^3\to S^2$ 
in a smooth way. We write the Hopf map $S^3 \to S^2$ using the quaternions,  $\eta(y)=y i\bar y$, 
where $y$ is a unit quaternion and the image of $\eta$ is contained in the unit purely imaginary quaternions. We also 
extend $\eta$ to all quaternions in the obvious way.
 
We write $S^8$ as the unit sphere of $\bb H \times \bb R \times \bb H$, and let $\psi:S^8 \to \bb R^8 \cong \bb H \times \bb H$ be given by 
\[
\psi
\begin{pmatrix}
x \\
\lambda \\
y
\end{pmatrix}
 = 
 \begin{pmatrix}
 x \\
 \lambda + \eta(y)
 \end{pmatrix}
\]
The image does 
not fall in the Euclidean sphere $S^7$. However since $|\psi(x)| = |x|^2 + \lambda^2 + |y|^4 \neq 0$, we can normalize 
and $\phi = \frac{\eta}{|\eta|} : S^8 \to S^7$ clearly suspends the Hopf map. 

\begin{rem}
There are other ways of building smooth suspensions of the Hopf map, which in principle would furnish different 
examples of the application of theorem \ref{main}. We present the simplest one. 
\end{rem}

 Then, the total space of the pull-back of $S^3\cdots Sp(2)\to S^7$ by $\phi$ is readily identified with the set
\[E^{11}=\left\{
\left(
\begin{array}{c}x\\ \lambda\\y\end{array}
\begin{array}{c}c\\ d\end{array}
\right) 
\in S^8\times S^7~\Big|~
\Big(\phi
\begin{pmatrix}
x \\
\lambda \\
y
\end{pmatrix},
\begin{pmatrix} 
c \\ 
d
\end{pmatrix}
\Big)
\in Sp(2) \, ,
\right\}
\]


The unit quaternions act on $E^{11}$ as follows:
\begin{equation}
\label{star 8}
 q\star\left(\begin{array}{c}qx\bar q\\ \lambda\\qy\end{array}\begin{array}{c}c\\ d\end{array}\right)=\left(\begin{array}{c}qx\bar q \\ \lambda\\qy\end{array}\begin{array}{c}qc\\ qd\end{array}\right)\in E^{11}\quad\text{if}\quad \left(\begin{array}{c}x \\ \lambda \\y\end{array}\begin{array}{c}c\\ d\end{array}\right)\in E^{11}\end{equation}

On one hand, the projection in the first column $pr_1:E^{11}\to S^8$ define it as the pull-back 
bundle of $Sp(2)\to S^7$, on the other hand $\star$ in \eqref{star 8} defines a new free action 
on $E^{11}$. The quotient of this action is diffeomorphic to the only exotic sphere in 
dimension 8, \cite{lohan-8}. Again, pulling back all the way from the Hopf map, we have that $E^{11}$ is the 
pullback of the Hopf bundle $S^3 \cdots S^7 \to S^4$ by $f = a \circ h \circ \phi$. 

We now have to study the inverse images by $f$ of its regular values. We have:

\begin{lem}
The only singular value of $f$ is the south pole $(-1,0)\in S^4\subset \bb R \times \bb H$. 
\end{lem}

\noindent {\it Proof of lemma} 
One can first note that the derivative of $\phi$ at $p=(1,0,0)^T$ is $D\phi_p(X,\Lambda,Y)=(X,\Lambda)$ which spans only one dimension transversal to the fiber $(z,0)\subset S^7$ making the point $(-1,0)\in S^4$ singular for $ah\phi$. On the other hand  if $y\neq 0$, $\phi$ is a submersion and therefore so is $f$, and the last case in hand is $p=(x,\lambda,0)$ which goes to a fiber different from $(z,0)$. In this last case, the differential of $\phi$ is again $D\phi_p(X,\Lambda,Y)=(X,\Lambda)$ however, now the space spanned by this differential is completely transversal to the fiber of $\phi(p)$ since the tangent space to the last is of the form $(x\xi,\lambda\xi)$, for purely imaginary quaternions $\xi$, in particular, it has no real part in the second coordinate.\qed

\medskip

Then we see that the inverse image of the singular point is $S^3 \subset S^8$ given by points $(x,0,0) \in S^8$. 
The inverse images of regular points are 4-dimensional submanifolds of $S^3$. In this case proposition \ref{stability}  
applies immediately, since there can be no embedding of a 4-dimensional manifold into a 3-dimensional one.  

\subsection{Kervaire spheres} \label{kervaire-bundles} Consider the principal bundle of special orthonormal 
frames over the round $S^{2n+1}$,
I.e., consider the principal bundle $SO(2n+1)\cdots SO(2n+2)\stackrel{pr}{\to} S^{2n+1}$ given by projection 
on the first column. Consider also $\tau:S^{2n}\to SO(2n+1)$ where $\tau(x)$ is defined as the 
reflection by the hyperplane orthogonal to $x$. As a map from $S^{4n+1}$ to $S^{2n+1}$ we can consider $J\tau$, 
defined as
\[J\tau(x,y)=\exp \tau\Big(\frac{y}{|y|}\Big)x\qquad\]
where $(x,y)\in\bb R^{2n+1}\times\bb R^{2n+1}$. This map extends continuously to $y=0$ and has an appropriate 
equivariant smoothing with same fibers (this situation is quite common in this kind of construction, see e.g. section 2.3 of 
\cite{DMR}).
Its homotopy type is the image of $\tau$ by the Hopf-Whitehead 
$J$-homomorphism.

The bundle $P=(J\tau)^*\pi\to S^{4n+1}$ admits a free $SO(2n+1)$ action (which is isometric with respect to 
natural metrics) with quotient that can be identified 
(using the techniques from \cite{lohan-8}) with $\Sigma^{4n+1}$, the Kervaire sphere of dimension $4n+1$
as presented in Chapter I, section 7 of \cite{bredon-book}. 
These examples are known to be exotic \cite{Hopkins-et-al} for infinitely many $n$'s.

\begin{equation}
\begin{xy}\xymatrix@R=8pt@C=2pt@1{  P\ar[rd]\ar[rr]\ar[dd] & & SO(2n+2)\ar[rd]\ar[dd]&\\ &\Sigma^{4n+1}\ar[rr] & &S^{2n+1}\\S^{4n+1}\ar[rr]^{J\tau}& &S^{2n+1}&}\end{xy}
\end{equation}

The fibers of the map 
$J\tau$ are the spheres $y=0$ and \[\{(\tau(y)^{-1}x,\lambda y)~|~y\in S^{2n},~\lambda\in[-1,0)\}\]
The information given by the dimensions and topology of the fibers is not sufficient to perceive an obstruction to 
Theorems \ref{main} and \ref{A-flat-totally-geodesic}, making $(J\tau)^*\pi$ a candidate to induce nonnegative 
curvature on $\Sigma^{4n+1}$.


\end{document}